\documentclass[a4paper,12pt]{amsart}
\usepackage{amsmath}
\usepackage{amssymb}
\usepackage{mathrsfs}
\usepackage{enumerate}
\usepackage{ifthen}
\usepackage{graphicx}
\usepackage{caption}
\usepackage[T1]{fontenc} 
\usepackage{tikz}
\usepackage{cite}

\nonstopmode \numberwithin{equation}{section}
\newtheorem{theorem}{Theorem}[section]

\newtheorem{corollary}{Corollary}[section]

\theoremstyle{remark}
\theoremstyle{definition}
\newtheorem{remark}{Remark}[section]
\newtheorem{definition}{Definition}[section]

\theoremstyle{plain}

\newtheorem*{lemA}{Lemma A}

\numberwithin{equation}{section}
\numberwithin{theorem}{section}

\newcounter{minutes}
\divide\time by 60
\newcounter{hours}
\multiply\time by 60
\addtocounter{minutes}{-\time}

\begin{document}

\title{Pre-Schwarzian and Schwarzian norm Estimates for class of Ozaki Close-to-Convex functions}

\author{Molla Basir Ahamed}
\address{Molla Basir Ahamed, Department of Mathematics, Jadavpur University, Kolkata-700032, West Bengal, India.}
\email{mbahamed.math@jadavpuruniversity.in}

\author{Rajesh Hossain}
\address{Rajesh Hossain, Department of Mathematics, Jadavpur University, Kolkata-700032, West Bengal, India.}
\email{rajesh1998hossain@gmail.com}

\subjclass[2020]{Primary 30C45, 30C55; Secondary 30C80,
31A10.}
\keywords{Analytic functions, Univalent functions, pre-Schwarzian and Schwarzian derivatives, Close-to-convex functions, Harmonic mappings, Norm estimates, Sharp results}

\def\thefootnote{}
\footnotetext{ {\tiny File:~\jobname.tex,
printed: \number\year-\number\month-\number\day,
          \thehours.\ifnum\theminutes<10{0}\fi\theminutes }
} \makeatletter\def\thefootnote{\@arabic\c@footnote}\makeatother

\begin{abstract}
The primary objective of this paper is to derive sharp bounds for the norms of the Schwarzian and pre-Schwarzian derivatives in the Ozaki close-to-convex functions $f$, expressed in terms of their value $f^{\prime\prime}(0)$, in particular, when the quantity is equal to zero. Additionally, we obtain sharp bounds for distortion and growth theorems. We will also derive the sharp bound of pre-Schwarzian norm for a certain class of harmonic mappings whose analytic part is fixed.
\end{abstract}
\maketitle
\pagestyle{myheadings}
\markboth{M. B. Ahamed and R. Hossain}{Pre-Schwarzian and Schwarzian norm Estimates for class of Ozaki Close-to-Convex functions}

\section{\bf Introduction}
Let $\mathcal{H}$ be the class of analytic functions in the unit disk $\mathbb{D}=\{z\in\mathbb{C} : |z|<1\}$, and $\mathcal{LU}$ denote the subclass of $\mathcal{H}$ consisting of all locally univalent functions, namely, $\mathcal{LU}=\{f\in\mathcal{H} : f^{\prime}(z)\neq 0, \; z\in\mathbb{D}\}$. The pre-Schwarzian and Schwarzian derivatives for a function $f\in\mathcal{LU}$ are defined by 
\begin{align*}
	P_f(z)=\frac{f^{\prime\prime}(z)}{f^{\prime}(z)}\;\; \mbox{and}\;\; S_f(z)=\left(\frac{f^{\prime\prime}(z)}{f^{\prime}(z)}\right)^{\prime}-\frac{1}{2}\left(\frac{f^{\prime\prime}(z)}{f^{\prime}(z)}\right)^2,
\end{align*}
respectively. It is well-known that $P(A\circ f)=Pf$ for all linear (or affine) transformations $A(w)=aw+b$, $b\neq 0$. Moreover,  the Schwarzian derivative is invariant under (non-constant) linear fractional (or M\"obius) transformations. In other words $S(T\circ f)=S(f)$ for every function $T$ of the form 
\begin{align*}
	T(w)=\frac{aw+b}{cw+d},\; ad-bc\neq 0.
\end{align*}
These are just particular cases of the chain rule for the pre-Schwarzian and Schwarzian derivatives: 
\begin{align*}
	P(g\circ f)=(Pg\circ f)f^{\prime}+Pf\;\; \mbox{and}\;\; S(g\circ f)=(Sg\circ f)\left(f^{\prime}\right)^2+Sf,
\end{align*}
respectively, which hold whenever the composition $g\circ f$ is defined.\vspace{2mm}

The pre-Schwarzian and Schwarzian derivatives of locally univalent analytic functions 
$f$ are essential tools for analyzing their geometric properties. They provide conditions for global univalence and impose geometric constraints on the range of $f$. Notably, any univalent analytic function $f$ in the unit disk $\mathbb{D}$ satisfies the sharp inequality
\begin{align*}
	|Pf(z)|\leq \frac{6}{1-|z|^2},\; z\in\mathbb{D}.
\end{align*}
Becker's univalence criterion (see \cite{Becker-JRAM-1983}) states that if \( f \) is locally univalent and analytic in \( \mathbb{D} \), then
\begin{align}\label{Eq-11.11}
	\sup_{z\in\mathbb{D}}|zPf(z)|\left(1-|z|^2\right)\leq 1.
\end{align}
Then \( f \) is univalent in the unit disk. In \cite{Becker-Pom-JRAM-1984}, Becker and Pommerenke later proved that the constant $1$ in \eqref{Eq-11.11}  is sharp.

\vspace{2mm} Also, the pre-Schwarzian and Schwarzian norms (the hyperbolic sup-norms) of $f\in\mathcal{LU}$ are defined by
\begin{align*}
	||P_f||=\sup_{z\in\mathbb{D}}\left(1-|z|^2\right)|P_f(z)|\;\;\mbox{and}\;\; ||S_f||=\sup_{z\in\mathbb{D}}\left(1-|z|^2\right)^2|S_f(z)|,
\end{align*}
respectively.\vspace{2mm}

These norms have significant meaning in the theory of Teichm\"uler spaces (see \cite{Lehto-1987}). For a univalent function $f\in\mathcal{LU}$, it is well known that $||P_f||\leq 6$ and $ ||S_f||\leq 6$ (see \cite{Kraus-MMSG-1932,Nehari-BAMS-1949}) and these estimates are best possible. On the other hand, for a locally univalent function $f\in\mathcal{LU}$, it is also known that if $||P_f||\leq 1$ (see \cite{Becker-Prommerenke-JRAM-1972,Becker-JRAM-1983}) or $||S_f||\leq 2$ (see \cite{Nehari-BAMS-1949}), then the function $f$ is univalent in $\mathbb{D}$. In $1976$, Yamashita (see \cite{Yamashita-MM-1976}) proved that $||P_f||$ is finite if, and only if, $f$ is uniformly locally univalent in $\mathbb{D}$. Furthermore, if $||P_f||\leq 2$, then $f$ is bounded in $\mathbb{D}$ (see \cite{Kim-Sugawa-RMJ-2002}). The pre-Schwarzian norm has been studied by many researchers (see, for example,  \cite{Yamashita-HMJ-1999,Sugawa-AUMCDS-1996,Ali-Pal-BDS-2023,Ali-Pal-MJM-2023,Ali-Pal-MM-2023,Ali-Pal-2023,Ali-Pal-PEMS-2023,Wang-Li-Fan-MM-2024,Carrasco-Hernández-AMP-2023} and references therein. Sharp norm estimate of Schwarzian derivative for a class of convex functions is established in \cite{Kanas-Sugawa-APM-2011}.\vspace{2mm}

Let us introduce one of the most important and useful tool known as differential subordination technique. In geometric function theory, many problems can be solved in a simple and sharp manner with the help of differential subordination. A function $f\in\mathcal{H}$ is said to be subordinate to another function $g\in\mathcal{H}$ if there exists an analytic function $\omega : \mathbb{D}\to\mathbb{D}$ with $\omega(0)=0$ such that $f(z)=g(\omega(z))$ and it is denoted by $f\prec g$. Moreover, when $g$ is univalent, then $f\prec g$ if, and only if, $f(0)=g(0)$ and $f(\mathbb{D}\subset g(\mathbb{D})$.\vspace{2mm}

Let $\mathcal{A}$ denote the subclass of $\mathcal{H}$ consisting of functions $f$ with normalized conditions $f(0)=f^{\prime}(z)-1=0$. Thus, any function $f$ in $\mathcal{A}$ has the Taylor series expansion of the form 
\begin{align}\label{Eq-1.1}
	f(z)=z+\sum_{n=2}^{\infty}a_nz^n\;\; \mbox{for all}\;\; z\in\mathbb{D}.
\end{align}
Let $\mathcal{S}$ be the subclass of $\mathcal{A}$ consisting of univalent (that is, one-to-one) functions. A function $f\in \mathcal{A}$ is called starlike (with respect to the origin) if $f(\mathbb{D})$ is starlike with respect to the origin and convex if $f(\mathbb{D})$ is convex. Let $\mathcal{S}^*(\alpha)$ and $\mathcal{C}(\alpha)$ denote respectively, the classes of starlike and convex functions of order $\alpha$ for $0\leq \alpha<1$ in $\mathcal{S}$. It is well-known that a function $f\in\mathcal{A}$ belongs to $\mathcal{S}^*(\alpha)$ if, and only if, ${\rm Re}(zf^{\prime}(z)/f(z))>\alpha$ for $z\in\mathbb{D}$, and $f\in\mathcal{C}(\alpha)$ if, and only if, ${\rm Re}(1+zf^{\prime\prime}(z)/f^{\prime}(z))>\alpha$. Similarly, a function $f\in\mathcal{A}$ belongs to $\mathcal{K}$, the class of close-to-convex functions, if and only if, there exists $g\in\mathcal{S}^*$ such that ${\rm Re}[e^{i\tau}(zf^{\prime}(z))/g(z)]>0$ for $z\in\mathbb{D}$ and $\tau\in (-\pi/2, \pi/2)$. Thus, $\mathcal{C}\subset\mathcal{S}^*\subset\mathcal{K}\subset\mathcal{S}$. In particular, when $\tau=0$, then the resulting subclass of the close-to-convex functions is denoted by $\mathcal{K}_0$.\vspace{2mm}

Although the class $\mathcal{K}$ was first formally introduced by Kaplan (see \cite{Kaplan-MMJ-1952}) in $1952$, already in $1941$, Ozaki (see \cite{Ozaki-SRTBD-1941}) considered functions $f$ in $\mathcal{A}$ satisfying the condition
\begin{align}\label{Eq-1.2}
	{\rm Re}\left(1+\frac{zf^{\prime\prime}(z)}{f^{\prime}(z)}\right)>-\frac{1}{2}\;\; \mbox{for}\;\; z\in\mathbb{D}.
\end{align}
It is important fact that the original definition of Kaplan see \cite{Kaplan-MMJ-1952} that functions satisfying \eqref{Eq-1.2} are close-to-convex and therefore, member of the class $\mathcal{S}$. \vspace{2mm}

In \cite{Kargar-Ebadian-SCMA-2017}, Kargar and Ebadian considered the following generalization to \eqref{Eq-1.2}.
\begin{definition}(see \cite[Definition 1.1]{Allu-Thomas-Tuneski-BAMS-2019})
	Let $f\in\mathcal{A}$ be a locally univalent for $z\in\mathbb{D}$ and let $-1/2<\lambda\leq 1$. Then $f\in\mathcal{F}(\lambda)$ if, and only if
	\begin{align}\label{Eq-1.3}
		{\rm Re}\left(1+\frac{zf^{\prime\prime}(z)}{f^{\prime}(z)}\right)>\frac{1}{2}-\lambda\;\; \mbox{for}\;\; z\in\mathbb{D}.
	\end{align}
\end{definition}
It is clear that, when $-1/2\leq\lambda\leq 1/2$, functions defined by \eqref{Eq-1.3} provide a subset of $\mathcal{C}$, with $\mathcal{F}(1/2)=\mathcal{C}$, and, since $1/2-\lambda\geq -1/2$ when $\lambda\leq 1$, functions in $\mathcal{F}(\lambda)$ are close-to-convex when $1/2\leq\lambda\leq 1$. We shall cal members $f$ of $\mathcal{F}(\lambda)$ when $1/2\leq\lambda\leq 1$ Ozaki close-to-convex functions and denote this class by $\mathcal{F}_0(\lambda)$. One important note that in contrast to the definition of the class $\mathcal{K}$, the definition of $\mathcal{F}(\lambda)$ does not involve an independent starlike function $g$, but, as was shown in \cite{Ponnuswamy-Sahoo-Yanagihari-NA-1963}, members of $\mathcal{F}(1)$ have coefficients which grow at the same rate as those in $\mathcal{K}$, that is, $\mathcal{O}(n)$ as $n\to\infty$.
\subsection{Pre-Schwarzian and Schwarzian derivatives of analytic functions.}
	
	The pre-Schwarzian and Schwarzian derivatives are important research tools in geometric function theory, especially to characterize Teichum\"uller space by using the pre-Schwarzian and Schwarzian derivatives embedding models. They are also powerful tools in the research of the inner radius of univalency for planer domains and quasiconformal extensions,  (see \cite{Lehto-1987,Lehto-JAM-1979}). The research of pre Schwarzian and Schwarzian derivatives has a long history and is closely related to other disciplines. As early as $1836$, Kummer firstly introduced the Schwarzian derivative when studying hypergeometric partial differential equations. Later, people have done a lot of research on the relationship between Schwarzian derivative and Univalent function. Several sufficient conditions for univalent analytic functions are obtained by using the notions of Pre-Schwarzian and Schwarzian derivatives.\vspace{2mm}
	
	 It is well known that the pre-Schwarzian norm $||Ph||\leq 6$ for the Univalent analytic function $h$ is defined in $\mathbb{D}$. In $1972$, Becker \cite{Becker-Prommerenke-JRAM-1972} used the pre-Schwarzian derivative to obtain the sufficient condition that the function in $\mathbb{D}$ is univalent, in other words, if $||Ph||\leq1$, then the function $h$ is univalent in $\mathbb{D}$. For different aspect of Schwarzian derivatives, we refer to the articles  \cite{Carrasco-Hernández-AMP-2011,Bhowmik-Wriths-CM-2012,Aghalary-Orouji-COAT-2014} and references therein.
	\subsection{Pre-Schwarzian and derivatives of Harmonic functions}
	A twice continuously differentiable complex-valued function $f=u+iv$ in a domain $\Omega$ is called harmonic if $u$ and $v$ both are harmonic in $\Omega$ or equivalently if it satisfies the Laplace equation $\Delta f=4f{_{z\bar{z}}}=0$. Every harmonic mapping $f$ has a canonical representation of the form $f=h+\bar{g}$, where $h$ and $g$ are analytic functions in $\Omega$ called the analytic and co-analytic  part of $f$, respectively. The jacobian of $f=h+\bar{g}$ is given by $J_{f}(z)=|f{_z}|^2-|f{_{\bar{z}}}|^2=|h^{\prime}(z)|^2-|g^{\prime}(z)|^2$. The harmonic mapping $f$ is called orientation-preserving or sense-preserving mapping if $J_{f}(z)>0$ and is called orientation-reversing or sense-reversing mapping if $J_{f}(z)<0$. For a sense-preserving harmonic mapping $f=h+\bar{g}$, its second complex dilatation is $\Omega$ with $\omega=g^{\prime}/h^{\prime}$ and $|\omega(z)<1|$ in $\Omega_1$. According to Lewy's theorem, see \cite{Lewy-BAMS-1936}, a harmonic mapping $f=h+\bar{g}$ is locally univalent in $\Omega$ if its Jacobian  $J_{f}(z)\neq0$. Stuyding the sufficient conditions for a locally univalent harmonic mapping to be univalent , as well as the necessary conditions for a univalent harmonic mapping is an interesting research area. For harmonic pre-Schwarzian and its applications, we refer to the articles \cite{Carrasco-Hernández-AMP-2023,Wang-Li-Fan-MM-2024,Hernández-Martín-JGA-2015,Hu-Fan-KMJ-2021,Liu-Ponnusamy-BSM-2019}.\vspace*{2mm}

 	Harmonic mappings play the natural role in parameterizing minimal surfaces in the context of differential geometry. Planner harmonic mappings have application not only in the differential geometry but also in the various field of engineering, physics, operations research and other intriguing aspects of applied mathematics. The theory of harmonic functions has been used to study and solve fluid flow problems (see \cite{aleman-2012}. The theory of univalent harmonic functions having prominent geometric properties like starlikeness, convexity and close-to-convexity appear naturally while dealing with planner fluid dynamical problems. For instance, the fluid flow problem on a convex domain satisfying an interesting geometric property has been extensively studied by Aleman and Constantin  \cite{aleman-2012}. With the help of geometric properties of harmonic mappings, Constantin and Martin \cite{constantin-2017} have obtained a complete solution of classifying all two dimensional fluid flows.\vspace*{1mm}
	
	Let $\mathcal{H}$ denotes the class of all harmonic mappings $f=h+\bar{g}$ in $\mathbb{D}$ with the normalization $h(0)=h^{\prime}(0)-1=g(0)=0$. The function $f=h+\bar{g}$ is of the form 
	\begin{align*}
		h(z)=z+\sum_{n=2}^{\infty}a_nz^n\;\;\mbox{and}\;\;g(z)=\sum_{n=1}^{\infty}b_nz^n.
	\end{align*}
	As harmonic mapping is the generalization of analytic function, there are many open problem and interesting fact in planer harmonic mappings, (see, e.g. \cite{Clunie-Sheil-Small-AASFMSAIM-1984,Duren-2004,Yamashita-MM-1976}).\vspace{2mm}
	
	In $2003$, Chuaqui \emph{et al.} \cite{Chuaqui-Duren-Osgood-JAM-2003} defined the Schwarzian derivative of harmonic mappings, if $\omega$ equals the square of an analytic functions. Without assuming any additional condition on $\omega$, in \cite{Hernández-Martín-JGA-2015}, Hern$\acute{a}$ndez and Mart$\acute{i}$n  gave another definition of Schwarzian derivative for the harmonic mappings. For locally univalent harmonic mapping  $f=h+\bar{g}$, $Sf$ is given by
	\begin{align*}
		Sf&=\left(\log J_{f}\right)_{zz}-\frac{1}{2}\left(\log J_{f}\right)_{z}\\&=Sh+\frac{\bar{\omega}}{1-|\omega|^2}\left(\frac{h^{\prime\prime}}{h^{\prime}}\omega-\omega^{\prime\prime}\right)-\frac{3}{2}\left(\frac{\omega^{\prime}\bar{\omega}}{1-|\omega|^2}\right)^2,
	\end{align*}
	where $Sh$ is the classical Schwarzian derivative of the analytic function $h$.\vspace{2mm} 
	
	The pre-Schwarzian derivative of harmonic mapping $f=h+\bar{g}$ defined as
	\begin{align*}
		Pf(z)=\left(\log J_{f}\right)_{z}=\frac{h^{\prime\prime}(z)}{h^{\prime}(z)}-\frac{\omega^{\prime}(z)\bar{\omega}(z)}{1-|\omega(z)|^2}.
	\end{align*}
\section{\textbf{Pre-Schwarzian and Schwarzian norm Estimates for Ozaki Close-to-Convex functions}}\label{Sec-2}
In this section, we firstly give the equivalent characterization of the class $\mathcal{F}_0(\lambda)$ (Ozaki close-to-convex class), next we present the distortion and growth theorem, and then we derive the results of pre-Schwarzian and Schwarzian norms for the class of $\mathcal{F}_0(\lambda)$  in terms of the value of $f^{\prime\prime}(0)$.
\begin{theorem}\label{Th-4.1}
	For $\frac{1}{2}\leq\lambda\leq1$, the following are equivalent:
	\begin{enumerate}
		\item[\emph{(i)}.] $f\in\mathcal{F}_0(\lambda)$.\vspace{2mm}
		
		\item[\emph{(ii)}.] $\displaystyle{\rm Re}\left(1+\frac{zf^{\prime\prime}(z)}{f^{\prime}(z)}\right)\geq\frac{1-2\lambda}{2}+\frac{1}{2}\left(\frac{1-|z|^2}{1+2\lambda}\right)\bigg|\frac{f^{\prime\prime}(z)}{f^{\prime}(z)}\bigg|^2$\;\;\;\mbox{for}\;\;$\displaystyle\lambda\in[1/2,1]$,\vspace{2mm}\\
		
		\item[\emph{(iii)}.] $\displaystyle\bigg|(1-|z|^2)\frac{f^{\prime\prime}(z)}{f^{\prime}(z)}-(1+2\lambda)\bar{z}\bigg|\leq1+2\lambda$.
	\end{enumerate}
 
\end{theorem}
\begin{proof}[\bf Proof of Theorem \ref{Th-4.1}]
	Firstly, we will prove $(i)$ is equivalent to $(ii)$.\\
	If  $f\in\mathcal{F}_0(\lambda)$, then there exists a Schwarz function $\omega: \mathbb{D}\rightarrow\mathbb{D}$ such that 
	\begin{align}\label{Eq-4.1}
		1+\frac{zf^{\prime\prime}(z)} {f^{\prime}(z)}= \frac{1+2\lambda\omega(z)}{1-\omega(z)}
	\end{align}
	holds with $\omega(0)=0$.\vspace{2mm} 
	
	Let $\omega(z)=z\phi(z)$ for some analytic function $\phi$ that satisfies $\phi(\mathbb{D})\subseteq\mathbb{D}$. Then it follows from  \eqref{Eq-4.1} that
	\begin{align*}
		\frac{f^{\prime\prime}(z)}{f^{\prime}(z)}=\frac{(1+2\lambda)\omega(z)}{z(1-\omega(z))}
	\end{align*}
	which shows that
	\begin{align}\label{Eq-4.2}
	\frac{f^{\prime\prime}(z)}{f^{\prime}(z)}=\frac{(1+2\lambda)\phi(z)}{(1-z\phi(z))}.
	\end{align}
	Consequently, we have
	\begin{align}\label{Eq-4.3}
		\phi(z)=\frac{\frac{f^{\prime\prime}(z)}{f^{\prime}(z)}}{(1+2\lambda)+\frac{zf^{\prime\prime}(z)}{f^{\prime}(z)}}.
	\end{align}
	Since $|\phi(z)|^2\leq1$, we have the inequality
	\begin{align*}
		\bigg|\frac{f^{\prime\prime}(z)}{f^{\prime}(z)}\bigg|^2\leq(1+2\lambda)^2+2(1+2\lambda)	{\rm Re}\left(\frac{zf^{\prime\prime}(z)}{f^{\prime}(z)}\right)+|z|^2\bigg|\frac{f^{\prime\prime}(z)}{f^{\prime}(z)}\bigg|^2.
	\end{align*}
	By factorizing, an easy computation gives that
	\begin{align}\label{Eq-4.22}
		(1+2\lambda)\left[(1+2\lambda)+2{\rm Re}\left(\frac{zf^{\prime\prime}(z)}{f^{\prime}(z)}\right)\right]\ge(1-|z|^2)\bigg|\frac{f^{\prime\prime}(z)}{f^{\prime}(z)}\bigg|^2
	\end{align}
	which implies that
	\begin{align*}
		2{\rm Re}\left(\frac{zf^{\prime\prime}(z)}{f^{\prime}(z)}\right)\geq-(1+2\lambda)+\left(\frac{1-|z|^2}{1+2\lambda}\right)\bigg|\frac{f^{\prime\prime}(z)}{f^{\prime}(z)}\bigg|^2.
	\end{align*}
	When $\lambda\in[\frac{1}{2}, 1]$, it is easy to see that
	\begin{align}\label{Eq-4.3}
	{\rm Re}\left(1+\frac{zf^{\prime\prime}(z)}{f^{\prime}(z)}\right)\geq\frac{1-2\lambda}{2}+\frac{1}{2}\left(\frac{1-|z|^2}{1+2\lambda}\right)\bigg|\frac{f^{\prime\prime}(z)}{f^{\prime}(z)}\bigg|^2.
	\end{align}
	Next, we will prove that $(ii)$ is equivalent to $(iii)$.\\
	By multiplying \eqref{Eq-4.22} by $(1-|z|^2)$ both side, we have
	\begin{align*}
		(1-|z|^2)^2\bigg|\frac{f^{\prime\prime}(z)}{f^{\prime}(z)}\bigg|^2\leq(1+2\lambda)^2(1-|z|^2)+2(1+2\lambda)(1-|z|^2){\rm Re}\left(\frac{zf^{\prime\prime}(z)}{f^{\prime}(z)}\right)
	\end{align*}
	which turns out that
	\begin{align*}
			(1-|z|^2)^2\bigg|\frac{f^{\prime\prime}(z)}{f^{\prime}(z)}\bigg|^2-2(1+2\lambda)(1-|z|^2){\rm Re}\left(\frac{zf^{\prime\prime}(z)}{f^{\prime}(z)}\right)+(1+2\lambda)^2|z|^2\leq(1+2\lambda)^2.
	\end{align*}
	Thus, we have the desired inequality
	\begin{align}\label{Eq-4.4}
	\bigg|	(1-|z|^2)\left(\frac{f^{\prime\prime}(z)}{f^{\prime}(z)}\right)-(1+2\lambda)\bar{z}\bigg|\leq(1+2\lambda).
	\end{align}
	Hence, complete the proofs.
\end{proof}
\begin{remark}
	We have the following results for the class of convex functions which are immediate from Theorem \ref{Th-4.1}.
\end{remark}
\begin{corollary}
	If $f\in\mathcal{F}_0(1/2)=\mathcal{C}$, then from \eqref{Eq-4.3},  we have
	\begin{align*}
		{\rm Re}\left(1+\frac{zf^{\prime\prime}(z)}{f^{\prime}(z)}\right)\geq\frac{1-2\lambda}{2}+\frac{1}{2}\left(\frac{1-|z|^2}{1+2\lambda}\right)\bigg|\frac{h^{\prime\prime}(z)}{h^{\prime}(z)}\bigg|^2
	\end{align*}
	which implies that
	\begin{align*}
	{\rm Re}\left(1+\frac{zf^{\prime\prime}(z)}{f^{\prime}(z)}\right)\geq\frac{1}{4}(1-|z|^2)\bigg|\frac{f^{\prime\prime}(z)}{f^{\prime}(z)}\bigg|^2.
	\end{align*}
\end{corollary}
\begin{corollary}
	If $f\in\mathcal{F}_0(1/2)=\mathcal{C}$, then we have  from \eqref{Eq-4.4}
	\begin{align*}
		\bigg|	(1-|z|^2)\left(\frac{f^{\prime\prime}(z)}{f^{\prime}(z)}\right)-(1+2\lambda)\bar{z}\bigg|\leq(1+2\lambda)
		\end{align*}
		which implies that
	\begin{align*}
		\bigg|(1-|z|^2)\frac{f^{\prime\prime}(z)}{f^{\prime}(z)}-2\bar{z}\bigg|\leq2.
	\end{align*}
\end{corollary}
We obtain the following result which is distortion theorem and growth theorem for the class $ \mathcal{F}_0(\lambda) $.
\begin{theorem}\label{Th-4.2}
	If $f\in\mathcal{F}_0(\lambda)$, then for all $z\in\mathbb{D}$ and $\frac{1}{2}\leq\lambda\leq1$,
	\begin{align*}
		\frac{1}{(1+|z|^2)^{\frac{1+2\lambda}{2}}}\leq|f^{\prime}(z)|\leq\frac{1}{(1-|z|^2)^{\frac{1+2\lambda}{2}}}
	\end{align*}
	and
	\begin{align*}
		\int_{0}^{|z|}\frac{1}{(1+\xi^2)^{\frac{1+2\lambda}{2}}} d|\xi|\leq|f(z)|\leq\int_{0}^{|z|}\frac{1}{(1-\xi^2)^{\frac{1+2\lambda}{2}}} d|\xi|.
	\end{align*}
\end{theorem}
\begin{proof}[\bf Proof of Theorem \ref{Th-4.2}]
	Let $f\in\mathcal{F}_0(\lambda)$. In view of \eqref{Eq-4.3}, we obtain $\phi(0)=0$. By using the Schwarz lemma, we have
	\begin{align*}
		\bigg|\frac{\frac{f^{\prime\prime}(z)}{f^{\prime}(z)}}{(1+2\lambda)+\frac{zf^{\prime\prime}(z)}{f^{\prime}(z)}}\bigg|^2\leq|z|^2
	\end{align*}
	which shows that
	\begin{align*}
	\bigg|\frac{f^{\prime\prime}(z)}{f^{\prime}(z)}\bigg|^2\leq(1+2\lambda)^2|z|^2+2(1+2\lambda)|z|^2{\rm Re}\left(\frac{zf^{\prime\prime}(z)}{f^{\prime}(z)}\right)+|z|^2\bigg|\frac{zf^{\prime\prime}(z)}{f^{\prime}(z)}\bigg|^2.
	\end{align*}
	Thus, we have
	\begin{align}\label{Eq-4.5}
		(1-|z|^4)\bigg|\frac{f^{\prime\prime}(z)}{f^{\prime}(z)}\bigg|^2\leq(1+2\lambda)^2|z|^2+2(1+2\lambda)|z|^2{\rm Re}\left(\frac{zf^{\prime\prime}(z)}{f^{\prime}(z)}\right).
	\end{align}
	Multiplying both side of \eqref{Eq-4.5} by $(1-|z|^4)$, we obtain
	\begin{align}\label{Eq-4.6}
		(1-|z|^4)^2\bigg|\frac{f^{\prime\prime}(z)}{f^{\prime}(z)}\bigg|^2-2(1+2\lambda)|z|^2(1-|z|^4){\rm Re}\left(\frac{zf^{\prime\prime}(z)}{f^{\prime}(z)}\right)\leq(1+2\lambda)^2|z|^2(1-|z|^4).
	\end{align}
	Adding both side $\left((1+2\lambda)|z|^2\bar{z}\right)^2$ in \eqref{Eq-4.6}, we see that
	\begin{align}\label{Eq-4.7}
	(1-|z|^4)^2&\bigg|\frac{f^{\prime\prime}(z)}{f^{\prime}(z)}\bigg|^2-2(1+2\lambda)|z|^2(1-|z|^4){\rm Re}\left(\frac{zf^{\prime\prime}(z)}{f^{\prime}(z)}\right)+(1+2\lambda)^2|z|^4\bar{z}\\&\nonumber\leq(1+2\lambda)^2|z|^2(1-|z|^4)+(1+2\lambda)^2|z|^4\bar{z}.
	\end{align}
	Multiplying both side of \eqref{Eq-4.7} by $|z|$, a simple calculation shows that
	\begin{align}
		\bigg|(1-|z|^4)\frac{zf^{\prime\prime}(z)}{f^{\prime}(z)}-(1+2\lambda)|z|^4\bigg|\leq(1+2\lambda)|z|^2
	\end{align}
	which implies that
	\begin{align}
	\frac{-(1+2\lambda)|z|^2}{1+|z|^2}\leq{\rm \rm Re}\left(\frac{zf^{\prime\prime}(z)}{f^{\prime}(z)}\right)\leq\frac{(1+2\lambda)|z|^2}{1-|z|^2}.
	\end{align}
	Let $z=re^{i\theta}$. Hence, 
	\begin{align}\label{Eq-4.10}
		\frac{-(1+2\lambda)r}{1+r^2}\leq \dfrac{\partial}{\partial r}\left(\log|f^{\prime}(re^{i\theta})|\right)\leq\frac{(1+2\lambda)r}{1-r^2}.
	\end{align}
When $1/2\leq\lambda\leq1$ integrating \eqref{Eq-4.10} w.r.t. $r$, we obtain
\begin{align}
		\frac{1}{(1+|z|^2)^{\frac{1+2\lambda}{2}}}\leq|f^{\prime}(z)|\leq\frac{1}{(1-|z|^2)^{\frac{1+2\lambda}{2}}}
	\end{align}
	Next, for the growth part of the theorem, from the upper bound it follows that
	\begin{align}
		|f^{\prime}(re^{i\theta})|=\bigg|\int_{0}^{r}f^{\prime}(re^{i\theta})e^{i\theta} dt\bigg|\leq\int_{0}^{r}|f^{\prime}(re^{i\theta})| dt\leq\int_{0}^{r}\frac{1}{(1-t^2)^{\frac{1+2\lambda}{2}}} dt
	\end{align}
	which implies 
	\begin{align}
		|f(z)|\leq\int_{0}^{|z|}\frac{1}{(1-\xi^2)^{\frac{1+2\lambda}{2}}} d|\xi|
	\end{align}
	for all $z\in\mathbb{D}$.\vspace{2mm}
	
	 It is well-known that if $f(z_0)$ is a point of minimum modulus on the image of the circle $|z|=r$ and $\gamma=f^{-1}(\Gamma)$, where $\Gamma$ is the line segment from $0$ to $f(z_0)$, then 
	\begin{align}
			|f(z)|\geq	|f(z_0)|\geq\int_{0}^{|z|}\frac{1}{(1+\xi^2)^{\frac{1+2\lambda}{2}}} d|\xi|.
	\end{align}
	This completes the proof.
\end{proof}
Now, we will find the sharp bound of the pre-Schwarzian and Schwarzian norms in terms of value $f^{\prime\prime}(0)$ in the class $\mathcal{F}_0(\lambda)$, under the assumption  that $f^{\prime\prime}(0)=0$. In order to prove our result, we need the following lemma, which we can find in the proof of \cite[Theorem 6]{Carrasco-Hernández-AMP-2023}.
\begin{lemA}\emph{\cite{Carrasco-Hernández-AMP-2023}}
	If $\phi(z):\mathbb{D}\rightarrow\mathbb{D}$ be analytic function, then 
	\begin{align}
		\frac{|\phi(z)|^2}{1-|\phi(z)|^2}\leq\frac{(\phi(0)+|z|)^2}{(1-|\phi(0)|)^2(1-|z|^2)|)}
	\end{align}
\end{lemA}
We derive results that provide sharp bounds for the pre-Schwarzian norm in terms of the value $f^{\prime\prime}(0)$ for the function $f\in\mathcal{F}_0(\lambda)$, particularly under the assumption $f^{\prime\prime}(0)=0$.
\begin{theorem}\label{Th-4.3}
 If $h\in\mathcal{F}_0(\lambda)$, then for all $z\in\mathbb{D}$ and $\frac{1}{2}\leq\lambda\leq1$, then
	\begin{align}
		||Pf||\leq(1+2\lambda).
	\end{align}
	The inequality is sharp.
\end{theorem}
\begin{proof}[\bf Proof of Theorem \ref{Th-4.3}]
	Since $\phi_1(z)=z\xi_1(z)$, with $|\xi_1(z)|<1$, then in \eqref{Eq-4.2}, we obtain
	\begin{align*}
		\sup_{z\in\mathbb{D}}(1-|z|^2)\bigg|\frac{f^{\prime\prime}(z)}{f^{\prime}(z)}\bigg|&\leq\sup_{z\in\mathbb{D}}(1-|z|^2)\frac{(1+2\lambda)|z\xi_1(z)|}{1-|z|^2|\xi_1(z)|}\\&\leq(1+2\lambda) \sup_{0\leq r\leq1}\frac{r(1-r^2)}{(1-r^2)}\\&=(1+2\lambda).
	\end{align*}
	Thus we desired bound is established. \vspace{1.2mm}
	
	The next part of the proof is to show that the bound is sharp. Henceforth, we consider the extremal function $f^*(z)$ is given by
	\begin{align*}
		f^*(z)=\int_{0}^{z}\frac{1}{(1-\xi^2)^{\frac{1+2\lambda}{2}}} d\xi \;\;\mbox{for all}\;\;\frac{1}{2}\leq\lambda\leq 1.
	\end{align*}
	It can be easily shown that $||Pf^*||=(1+2\lambda)$.
\end{proof}
From Theorem \ref{Th-4.3}, we obtain the following immediate result for the class $\mathcal{C}$ of convex functions.
\begin{corollary}
		If $f\in\mathcal{F}_0(1/2)=\mathcal{C}$ then for all $z\in\mathbb{D}$, we have
	\begin{align*}
		||Pf||\leq 2.
	\end{align*}
	The bound is sharp.
	
\end{corollary}
\subsection*{Sharpness of Corolary 4.3 :}
For $\lambda=1/2$, it follows from \eqref{Eq-4.32} that

\begin{align*}
	\frac{f^{\prime\prime}_\frac{1}{2}}{f_\frac{1}{2}^{\prime}}(z)=\frac{2z}{1-z^2}\;\;\mbox{and}\;\;Pf_\frac{1}{2}=\frac{2z}{1-z^2}.
\end{align*}
A simple computation thus yields that
\begin{align*}
	||Pf_\frac{1}{2}||=\sup_{z\in\mathbb{D}}\left(1-z^2\right)|Pf_\frac{1}{2}|=\sup_{z\in\mathbb{D}}\left(1-|z|^2\right)\frac{2|z|}{1-|z|^2}=2
\end{align*}
  and we see the constant $2$ is sharp.
\begin{remark}
	If $f\in\mathcal{F}_0(\frac{1}{2})=\mathcal{C}$, then for all $z\in\mathbb{D}$, we have
	\begin{align*}
		(1-|z|^2)\bigg|\frac{h_1^{\prime\prime}(z)}{h_1^{\prime}(z)}\bigg|\leq2|z|.
	\end{align*}
\end{remark}
Our next result gives a sharp bound for the norm of the Schwarzian derivative when $f\in \mathcal{F}_0(\lambda)$ by a direct application of the Schwarz lemma.
\vspace*{2mm}
\begin{theorem}\label{Th-4.4}
		If $f\in\mathcal{F}_0(\lambda)$, for all $z\in\mathbb{D}$ and $\frac{1}{2}\leq\lambda\leq1$, then the Schwarzian norm 
		\begin{align*}
			||Sf_\lambda||=(1-|z|^2)^2|Sh(z)|\leq\frac{(1+2\lambda)(3-2\lambda)}{2}.
		\end{align*}
		The inequality is sharp.
	\end{theorem}
	\begin{proof}[\bf Proof of Theorem \ref{Th-4.4}]
		From \eqref{Eq-4.2}, we have
		\begin{align*}
			\left(\frac{f^{\prime\prime}(z)}{f^{\prime}(z)}\right)=\frac{(1+2\lambda)\phi(z)}{(1-z\phi(z))}
		\end{align*}
		A simple calculation shows that
		\begin{align}
			Sf(z)=(1+2\lambda)\frac{\phi^{\prime}(z)+(\frac{1-2\lambda}{2})\phi^2(z)}{(1-z\phi(z))^2}.
		\end{align}
		By using triangle inequality and Schwarz pick lemma, we obtain
		\begin{align}\label{Eq-4.20}
			(1-|z|^2)^2|Sf|&\leq(1+2\lambda)\bigg|\phi^{\prime}(z)+\left(\frac{1-2\lambda}{2}\right)\phi^2(z)\bigg|\frac{(1-|z|^2)^2}{|1-z\phi(z)|^2}\\&=\nonumber\frac{(1+2\lambda)(1-|z|^2)^2}{|1-z\phi(z)|^2}\left(\frac{1-|\phi(z)|^2}{1-|z|^2}+\left(\frac{1-2\lambda}{2}\right)|\phi(z)|^2\right).
		\end{align}
		We define the function $\Phi(z):\mathbb{D}\rightarrow\mathbb{D}$ such that
		\begin{align}
			\Phi(z)=\frac{\bar{z}-\phi(z)}{1-z\phi(z)}.
		\end{align}
		Since $\phi(\mathbb{D})\subseteq\mathbb{D}$, then we have $(1-|z|^2)(1-|z\phi(z)|^2)>0$, hence it follows that 
		\begin{align}
			|\bar{z}-\phi(z)|^2<|1-z\phi(z)|^2.
		\end{align}
		Thus, we conclude that $|\Phi(z)|^2<1$. \vspace{2mm}
		
		Consequently, we have the estimates
		\begin{align}
			1-|\Phi(z)|^2=\frac{(1-|\phi(z)|^2)(1-|z|^2)}{|1-z\phi(z)|^2}
		\end{align}
		and
		\begin{align}\label{Eq-4.24}
			\frac{(1-|z|^2)^2}{|1-z\phi(z)|^2}=\frac{(1-|\Phi(z)|^2)(1-|z|^2)}{(1-|\phi(z)|^2)}.
		\end{align}
		If we replace the expression \eqref{Eq-4.24} in \eqref{Eq-4.20}, we have
		\begin{align}\label{Eq-4.21}
			(1-|z|^2)^2|Sf(z)|\leq(1+2\lambda)(1-|\Phi(z)|^2)\left(1+(\frac{1-2\lambda}{2})\frac{|\phi(z)|^2(1-|z|^2)}{(1-|\phi(z)|^2)}\right).
		\end{align}
		Since $f^{\prime\prime}(0)=0$ implies that $\phi(0)=0$, using Lemma A, we easily obtain that
		\begin{align}\label{Eq-4.26}
			\frac{|\phi(z)|^2}{1-|\phi(z)|^2}\leq\frac{|z|^2}{1-|z|^2}.
		\end{align}
		In view of  \eqref{Eq-4.26} and \eqref{Eq-4.21}, we obtain
		\begin{align}
			(1-|z|^2)^2|Sf(z)|\leq(1+2\lambda)(1-|\Phi(z)|^2)\left(1+(\frac{1-2\lambda}{2})|z|^2\right).
		\end{align}
		Also, becouse $1-|\Phi(z)|^2$$\leq$1, we see that
		\begin{align}
				(1-|z|^2)^2|Sf(z)|&\leq(1+2\lambda)\left(1+\frac{1-2\lambda}{2}\right)\\&\nonumber=\frac{(1+2\lambda)(3-2\lambda)}{2}.
		\end{align}
		Thus our desired inequality is established.\vspace{1.2mm} 
		
		Next part of the proof is to show the sharpness of the inequality.	Henceforth, we consider the function $f_\lambda (z)$ is given by
	\begin{align}
		f_\lambda (z)=\int_{0}^{z}\frac{1}{(1-\xi^2)^{\frac{1+2\lambda}{2}}} d\xi, \;\;\;\;\;\mbox{for all} \;\frac{1}{2}\leq\lambda\leq1
	\end{align}
	maximizes the Schwarzian norm defined as:
	\begin{align*}
		||Sf||=\sup_{z\in\mathbb{D}}(1-|z|^2)^2|Sf| 
	\end{align*}
	and from this, the sharpness of the inequality holds for $\frac{1}{2}\leq\lambda\leq1$.\vspace{2mm} 
	
	Note that
	\begin{align}\label{Eq-4.32}
		\frac{f^{\prime\prime}_\lambda}{f_\lambda^{\prime}}(z)=\frac{(1+2\lambda)z}{1-z^2}\;\;\mbox{and}\;\;Sf_\lambda=\frac{(1+2\lambda)}{(1-z^2)^2}\left[1+\left(\frac{1-2\lambda}{2}\right)|z|^2\right]
	\end{align}
	which calculates
	\begin{align*}
		||Sf_\lambda||=\sup_{z\in\mathbb{D}}(1-|z|^2)^2|Sf_\lambda|=(1+2\lambda)\frac{(3-2\lambda)}{2}.
	\end{align*}
	In general, the integral formula for $f_\lambda$ given in above does not give primitives in terms of elementary functions, however when $\lambda=1/2$, we see that
	\begin{align*}
			f_\frac{1}{2} (z)=\int_{0}^{z}\frac{1}{(1-\xi^2)} d\xi=\frac{1}{2}\log\left(\frac{1+z}{1-z}\right),
	\end{align*}
	where $||Sf_\frac{1}{2}||=2$.
		\end{proof}
For the class $\mathcal{C}$, we have the an immediate result from Theorem \ref{Th-4.4}.
\begin{corollary}
		If $f\in\mathcal{F}_0(\frac{1}{2})=\mathcal{C}$, then for all $z\in\mathbb{D}$, we have
		\begin{align*}
			||Sf_\frac{1}{2}||\leq2.
		\end{align*}
		The inequality is sharp.
\end{corollary}
For the value $|f^{\prime\prime}(0)|$ which is not necessarily zero, we will give a bound for the quantity $(1-|z|^2)^2|Sf(z)|$ with $f\in \mathcal{F}_0(\lambda)$ in the following.
\begin{theorem}\label{Th-4.5}
	If $f\in\mathcal{F}_0(\lambda)$, for all $z\in\mathbb{D}$ and $\frac{1}{2}\leq\lambda\leq1$ and $\gamma=|\phi(0)|=\frac{|f^{\prime\prime}(0)|}{(1+2\lambda)}$, then 
	\begin{align*}
		(1-|z|^2)^2|Sf(z)|\leq(1+2\lambda)\left(1+\frac{1-2\lambda}{2}\frac{1+\gamma}{1-\gamma}\right).
	\end{align*}
\end{theorem}
\begin{proof}[\bf Proof of Theorem \ref{Th-4.5}]
	Let $\gamma=|\phi(0)|$. In view of the Lemma A, we  see that
	\begin{align}\label{Eq-4.27}
			\frac{|\phi(z)|^2}{1-|\phi(z)|^2}\leq\frac{(\gamma+|z|)^2}{(1-\gamma^2)(1-|z|^2)}.
	\end{align}
	If we substitute in \eqref{Eq-4.27} in \eqref{Eq-4.21}, then we obtain
	\begin{align*}
		(1-|z|^2)^2|Sf(z)|\leq(1+2\lambda)(1-|\Phi(z)|^2)\left(1+(\frac{1-2\lambda}{2})\frac{(\gamma+|z|)^2}{(1-\gamma^2)}\right).
	\end{align*}
	From the fact that $|z|<1$ and  $1-|\Phi(z)|^2\leq1$, we easily calculate that 
	\begin{align*}
		(1-|z|^2)^2|Sf(z)|\leq(1+2\lambda)\left(1+\frac{1-2\lambda}{2}\frac{1+\gamma}{1-\gamma}\right)
	\end{align*}
	which is the required inequality. This completes the proof.
\end{proof}
\section{\textbf{Schwarzian derivatives for Harmonic mapping with fixed analytic part}}\label{Sec-3}
In this article, we are concerned with a class of functions $\mathcal{G}(\beta)$, $\beta>0$, defined by
\begin{align*}
	\mathcal{G}(\beta)=\bigg\{f\in\mathcal{A}:{\rm Re}\left(1+\frac{zf^{\prime\prime}(z)}{f^{\prime}(z)}\right)<1+\frac{\beta}{2}\;\;\mbox{for}\;z\in\mathbb{D}\bigg\}.
\end{align*}
In $1941$, Ozaki \cite{Ozaki-SRTB-1941} introduced the class $\mathcal{G}:=\mathcal{G}(1)$ and proved that functions in $\mathcal{G}$ are univalent in $\mathbb{D}$. Later on, Umezawa \cite{Umezawa-JMSJ-1952} studied the class $\mathcal{G}$ and showded that functions in $\mathcal{G}$ are convex in one direction, $\emph{i.e.,}$ every $f\in\mathcal{G}$ maps $|z|=\rho<r$ for every $\rho$ near $r$ into a contour which may be cut by every straight line parallel too this direction in not more than two points. Moreover, functions in $\mathcal{G}$ are starlike in $\mathbb{D}$ (see \cite{Jovanovi´c-Obradovi´c-F-1995,Ponnuswamy-Rajasekaran-SJM-1995} ) . Thus, the class $\mathcal{G}(\beta)$ is included in $\mathcal{S}^*$ whenever $\beta\in(0,1]$. One can easily show that functions in $\mathcal{G}(\beta)$ are not univalent in $\mathbb{D}$ for $\beta>1$. $0<\beta\leq2/3$, the class $\mathcal{G}(\beta)$ was studied by Uralegaddi \textit{et al.} in \cite{Uralegaddi-Ganigi-Sarangi-TJM-1994}. \vspace*{2mm}

In this section, our main aim to derive estimates for the modulus of the pre-Schwarzian norm for the functions in $\mathcal{G}(\beta),\;\beta>0$. These result will provide a precise estimate of the pre-Schwarzian norm for functions $\mathcal{G}(\beta)$.
\begin{theorem}\label{Th-3.1}
	If $f=h+\bar{g}\in\mathcal{G}(\beta)$, then $||P_{f}||\leq2\beta+1$. The estimate is best possible.
\end{theorem}
\begin{proof}[\bf Proof of Theorem \ref{Th-3.1}]
	Let, $f=h+\bar{g}\in\mathcal{G}(\beta)$. Then $f$ holds the subordination relation 
	\begin{align*}
		1+\frac{zf^{\prime\prime}(z)}{f^{\prime}(z)}\prec\frac{1-(1+\beta)z}{1-z},
	\end{align*}
	which gives us
	\begin{align}\label{Eq-5.2}
		\bigg|\frac{f^{\prime\prime}(z)}{f^{\prime}(z)}\bigg|\leq\frac{\beta}{1-|z|}.
	\end{align}
	By the Schwarz-Pick lemmma and \eqref{Eq-5.2}, we have 
	\begin{align*}
		||P_{f}||&=\sup_{z\in\mathbb{D}}\left(1-|z|^2\right)|P_{f}|\\&=\sup_{z\in\mathbb{D}}\left(1-|z|^2\right)\bigg|\frac{f^{\prime\prime}(z)}{f^{\prime}(z)}-\frac{\bar{\omega}(z)\omega^{\prime}(z)}{1-|\omega(z)|^2}\bigg|\\&\leq\sup_{z\in\mathbb{D}}\left(1-|z|^2\right)\left(\bigg|\frac{f^{\prime\prime}(z)}{f^{\prime}(z)}\bigg|+\bigg|\frac{\bar{\omega}(z)\omega^{\prime}(z)}{1-|\omega(z)|^2}\bigg|\right)\\&\leq\sup_{z\in\mathbb{D}}\left(1-|z|^2\right)\left(\frac{\beta}{1-|z|}+\frac{|\bar{\omega}(z)|}{1-|z|^2}\right)\\&\leq\sup_{z\in\mathbb{D}}\left(\beta(1+|z|)+|\omega(z)|\right)\\&=2\beta+1.
	\end{align*}
	We now show that the estimate is best possible. For
 \[ \begin{cases}
 	\sqrt{1-\beta}\leq t<1\;\;\;\mbox{when}\;\;\beta\in[0,1)\vspace{2mm}\\\dfrac{1-2\alpha}{3-2\alpha}\leq t<1\;\;\;\;\;\;\;\;\mbox{when}\;\;\beta\in[1,\infty),
 \end{cases}
 \]
	we consider the function $f=h+\bar{g}\in\mathcal{G}(\beta)$ with the second complex dilatation $\omega_t(z)=\frac{z-t}{1-tz}$ and $h(z)$ be such that
	\begin{align*}
		1+\frac{zh_1^{\prime\prime}(z)}{h_1^{\prime}(z)}=\frac{1-(1+\beta)z}{1-z}.
	\end{align*}
	Then clearly $f\in\mathcal{G(\beta)}$ for all t in $[0,\infty)$. A simple calculation gives that
	\begin{align*}
		\frac{\bar{\omega_t}(z)\omega_t(z)}{1-|\omega_t(z)|^2}=\frac{\bar{z}-t}{(1-tz)(1-|z|^2)}.
	\end{align*}
	Consequently, we have
	\begin{align*}
		||P_{f_t}||&=\sup_{z\in\mathbb{D}}(1-|z|^2)\bigg|\frac{h^{\prime\prime}_1(z)}{h^{\prime}_1(z)}-\frac{\bar{\omega}(z)\omega^{\prime}(z)}{1-|\omega(z)|^2}\bigg|\\&=\sup_{z\in\mathbb{D}}(1-|z|^2)\bigg|\frac{\beta}{1-|z|}-\frac{\bar{z}-t}{(1-tz)(1-|z|^2)}\bigg|.
	\end{align*}
	We set
	\begin{align*}
		M_t=\sup_{z\in\mathbb{D}}(1-|z|^2)\bigg|\frac{\beta}{1-|z|}-\frac{\bar{z}-t}{(1-tz)(1-|z|^2)}\bigg|
	\end{align*}
	\begin{align}\label{Eq-5.4}
		=\sup_{z\in[0,1)}\bigg|\beta(1+r)-\frac{r-t}{1-tr}\bigg|=\sup_{z\in[0,1)}\;\phi(r),
	\end{align}
	where 
	\begin{align*}
		\phi(r)=\beta(1+r)-\frac{r-t}{1-tr}.
	\end{align*}
	A simple computation shows that
	\begin{align*}
		\phi^{\prime}(r)=\beta-\frac{1-t^2}{(1-rt)^2}\;\;\mbox{and}\;\;\phi^{\prime\prime}(r)=-\frac{2t(1-t^2)(1-rt)}{(1-rt)^4}<0\;\;\mbox{for all r}\;\in[0,1).
	\end{align*}
	Now, $\phi^{\prime}(r)=0$ gives
	\begin{align*}
		r=r_0:=\frac{1}{t}-\frac{1}{t}\sqrt{\frac{1-t^2}{\beta}}.
	\end{align*}
	Thus the maximum value of $\phi$ is attained at $r_0$. Hence, we have
	\begin{align}\label{Eq-5.5}
		M_t=\phi(r_0)=\frac{1+2\beta-2\beta\sqrt{\frac{1-t^2}{\beta}}}{t}.
	\end{align}
	In view of \eqref{Eq-5.4} and \eqref{Eq-5.5}, it is clear that $M_t\leq||P_{f_t}||\leq2\lambda+1$. We note that $M_t$ is increasing function for t and $M_t\rightarrow1+2\beta$ as $t\rightarrow1$. This shows that estimate is the best possible.
\end{proof}

\noindent\textbf{Compliance of Ethical Standards:}\\

\noindent\textbf{Conflict of interest.} The authors declare that there is no conflict  of interest regarding the publication of this paper.\vspace{1.5mm}

\noindent\textbf{Data availability statement.}  Data sharing is not applicable to this article as no datasets were generated or analyzed during the current study.

    \end{document}